\documentclass[a4paper,10pt]{amsart}
\usepackage[shortlabels]{enumitem}
\usepackage{amsthm,amsmath,amssymb,graphics,graphicx,hyperref,epstopdf,mathrsfs}
\usepackage{mathtools}
\usepackage{breqn}      
\usepackage{tikz}  
\usetikzlibrary{shapes.geometric}

\usetikzlibrary{calc,intersections,through}
  \usepackage{capt-of}

\title{The ErdŐs-Faber-Lovász conjecture for weakly dense hypergraphs}
\author{Guillermo Alesandroni}
\address{2000 Rosario, Santa Fe, Argentina}
\email{guillea@okstate.edu, alesangc@wfu.edu, alesandronig@yahoo.com}

\newtheorem{theorem}{Theorem}[section]

\newtheorem{lemma}[theorem]{Lemma}

\newtheorem{conjecture}[theorem]{Conjecture}

\theoremstyle{definition}
\newtheorem{definition}[theorem]{Definition}

\DeclareMathOperator{\adj}{adj}

\begin{document}
\maketitle
\begin{abstract}
Generalizing the concept of dense hypergraph, we say that a hypergraph with n edges is weakly dense, if no $k$ in the half-open interval $[2,\sqrt{n})$ is the degree of more than $k^2$ vertices. In our main result, we prove the famous Erdős-Faber-Lovász conjecture when the hypergraph is weakly dense.
\end{abstract}

\section{Introduction}

In one of its various equivalent versions, the Erdős-Faber-Lovász (EFL) conjecture [Er] reads: a linear $n$-uniform hypergraph with $n$ edges can be colored with $n$ colors.

In his short but ingenious article [SA], A. Sanchez-Arroyo defines the term dense hypergraph, and shows that the intriguing EFL conjecture holds under the hypothesis of density. Inspired by his work, we introduce the more general notion of weakly dense hypergraph, and prove the conjecture when the property of weak density is satisfied.

The concepts of density and weak density are defined and related as follows: a hypergraph $\mathscr{H}$ with $n$ edges is 
\begin{enumerate}[(1)]
\item \textbf{dense}, if no $k$ in the interval $[2,\sqrt{n}]$ is the degree of a vertex;
\item \textbf{slightly weakly dense}, if no $k$ in the interval $[2,\sqrt{n})$ is the degree of a vertex;
\item \textbf{weakly dense}; if no $k$ in the interval $[2,\sqrt{n})$ is the degree of more than $k^2$ vertices.
\end{enumerate}

(The definition of density given in [SA] differs slightly from ours, the difference being a minor technicality. However, we favor our definition because, in its strongest form, the theorem in [SA] proves the EFL conjecture for dense hypergraphs, with our notion of density.)

Note that $\{$dense hypergraphs$\}\subseteq\{$slightly weakly dense hypergraphs$\}\subseteq\{$weakly dense hypergraphs$\}$.

At the time of this writing, the conjecture has eluded proof for about half a century, and the reason may be that it is false. The class of weakly dense hypergraphs sheds some light on how to construct a counterexample. If it exists, a counterexample must contain more than $k^2$ vertices of degree $k$, for some $k\in [2,\sqrt{n})$. 

\section{Background and Notation}

In this section, we include a brief glossary of hypergraph terminology, and we state the EFL conjecture to make the material precise and self-contained.

\begin{definition}
\end{definition}

\begin{itemize}
\item A \textbf{hypergraph} $\mathscr{H}$ is a pair $\mathscr{H}=(\mathscr{V},\mathscr{E})$, where $\mathscr{E}$ is a finite family of nonempty sets, and $\mathscr{V}=\bigcup\limits_{E\in \mathscr{E}} E$. The elements of $\mathscr{V}$ are called \textbf{vertices}, and the elements of $\mathscr{E}$ are called \textbf{edges}. Two vertices belonging to a common edge are called \textbf{adjacent}, or \textbf{incident}, and an edge $E$ that contains a vertex $v$ is said to be \textbf{incident} to $v$.
\item We say that $\mathscr{H}$ is \textbf{linear} if any two edges share at most one vertex, and is \textbf{n-uniform} if each edge has exactly $n$ vertices.
\item The \textbf{degree} in $\mathscr{H}$ of a vertex $v$ (or just the degree of $v$), denoted $d_{\mathscr{H}}(v)$ (or just $d(v)$), is the number of edges containing $v$. The \textbf{minimum degree} of $\mathscr{H}$, denoted $\delta(\mathscr{H})$, is defined as $\delta(\mathscr{H})=\min_{v\in\mathscr{V}}d_{\mathscr{H}}(v)$.
\item A hypergraph $\mathscr{H}$ with $n$ edges is called \textbf{dense}, if no integer $k$ in the closed interval $[2,\sqrt{n}]$ is the degree of a vertex of $\mathscr{V}$. In other words, $\mathscr{H}$ is dense if $d(v)=1$ or $d(v)>\sqrt{n}$, for all $v\in\mathscr{V}$.
\item If $\mathscr{H}=(\mathscr{V},\mathscr{E})$, and $\mathscr{V}'\subseteq \mathscr{V}$, we define a \textbf{$k$-coloring} of $\mathscr{V}'$ as a function $f:\mathscr{V}' \rightarrow \{0,\ldots,k-1\}$, such that if $v,v' \in \mathscr{V}' \cap E$ for some edge $E\in\mathscr{E}$, then $f(v)\neq f(v')$. If a k-coloring $f$ of $\mathscr{V}'$ exists, the set $\{0,\ldots,k-1\}$ is referred to as a set of colors; a vertex $v\in \mathscr{V}'$ is said to be colored with the color $f(v)$, and we say that $\mathscr{V}'$ can be colored with $k$ colors. In particular, when $\mathscr{V}' = \mathscr{V}$, $f$ is called a \textbf{$k$-coloring} of $\mathscr{H}$, and we say that $\mathscr{H}$ can be \textbf{colored} with $k$ colors.
\item The \textbf{chromatic number} of $\mathscr{H}$, denoted $\chi(\mathscr{H})$, is the smallest value of $k$ such that a $k$-coloring of $\mathscr{H}$ exists. 
\end{itemize}

There is no universally accepted definition of $k$-coloring of a hypergraph, which explains the need for this glossay. The definition of $k$-coloring of our choice is the one that best suits our formulation of EFL conjecture. The renowned EFL conjecture states the following.

\begin{conjecture}
If $\mathscr{H}$ is a linear $n$-uniform hypergraph with $n$ edges, then $\chi(\mathscr{H})=n$.
\end{conjecture}

\section{Main Results}

The EFL conjecture for weakly dense hypergraphs relies on a series of lemmas and theorems that we consider below.

\begin{lemma}\label{Lemma 1}
Let $\mathscr{H} = (\mathscr{V},\mathscr{E})$ be a linear hypergraph with at most $n$ edges, each with at most $n$ vertices. Suppose that $\delta(\mathscr{H})\geq \sqrt{n}$, and let $E\in \mathscr{E}$. Then,
\begin{enumerate}[(i)]
\item $\mid E\mid\leq \sqrt{n}+1$,
\item $\mid E\mid <\sqrt{n}+1$ if, for some $v \in E$,  $d(v)\geq \sqrt{n}+1$.
\end{enumerate}
\end{lemma}

\begin{proof}
Let $E=\{v,c_1,\ldots,c_k\}$. We will show that $k\leq\sqrt{n}$. By the linearity of $\mathscr{H}$, there is an edge (namely, $E$), that contains both $v$ and $c_1$, and there are at least $\sqrt{n}-1$ edges containing $c_1$ but not $v$. These edges are among the at most $n-\sqrt{n}$ edges not incident to $v$. Hence, there are at most $(n-\sqrt{n})-(\sqrt{n}-1)$ edges not containing either of $c_1$ and $v$ (and this upper bound is not reached if $ d(v)>\sqrt{n}$ or $d(c_1)>\sqrt{n}$). Let $2\leq i \leq k$, and suppose that the number of edges not containing any of $v,c_1,\ldots,c_{i-1}$ is at most $(n-\sqrt{n})-(i-1)(\sqrt{n}-1)$ (and this upper bound is not reached if one of $v,c_1,\ldots,c_{i-1}$ has degree larger than $\sqrt{n}$). By the linearity of $\mathscr{H}$, there is one edge (namely, $E$), containing each of $v,c_1,\ldots,c_i$ and at least $\sqrt{n}-1$ edges containing $c_i$ but not contining any of $v,c_1,\ldots,c_{i-1}$. Therefore, the number of edges not containing any of $v,c_1,\ldots,c_i$ is at most $(n-\sqrt{n})-i(\sqrt{n}-1)$ (and this upper bound is not reached if one of $v,c_1,\ldots,c_i$ has degree larger than $\sqrt{n}$). By induction on $i$, we have proved that the number of edges not containing any of $v,c_1,\ldots,c_k$ is at most $(n-\sqrt{n})-k(\sqrt{n}-1)$ (and this upper bound is not reached if any of $v,c_1,\ldots,c_k$ has degree larger than $\sqrt{n}$). Therefore, $k\leq\sqrt{n}$ (and $k<\sqrt{n}$ if any of $v,c_1,\ldots,c_k$ has degree $\geq \sqrt{n}$).

(i) Since $k\leq \sqrt{n}$, $\mid E\mid =k+1\leq \sqrt{n}+1$.

(ii) If one of $v,c_1,\ldots,c_k$ has degree larger than $\sqrt{n}$, then $k<\sqrt{n}$, and $\mid E\mid =k+1<\sqrt{n}+1$.
\end{proof}

\begin{definition}
We define the \textbf{adjacency} $\adj(v)$ \textbf{of a vertex} $v$ of $\mathscr{H}$ as 
\[\adj(v)=\{v'\in \mathscr{V}: v' \text{ is adjacent to }v\}.\] 
\end{definition}

\begin{lemma}\label{Lemma 1'}
Let $\mathscr{H}=(\mathscr{V},\mathscr{E})$ be a linear hypergraph with at most $n$ edges, each with at most $n$ vertices. Suppose that $\delta(\mathscr{H})\geq \sqrt{n}$. Suppose, in addition, that there is a vertex $v\in \mathscr{V}$, such that $d(v)=\sqrt{n}$, and $\mid \adj(u)\mid =n$ for all $u\in \adj(v)\cup\{v\}$. Then,
\begin{enumerate}[(i)]
\item $\mid \mathscr{E} \mid = n$;
\item $\mid E\mid=\sqrt{n}+1$, for all $E\in \mathscr{E}$;
\item $d(u)=\sqrt{n}$, for all $u \in \mathscr{V}$;
\item every pair of edges share exactly one vertex;
\item if $E\in \mathscr{E}$, and $u\in \mathscr{V}\setminus E$, there is exactly one vertex in $E$ not adjacent to $u$.
\end{enumerate}
\end{lemma} 

\begin{proof}
Let $\{E_1,\ldots,E_{\sqrt{n}}\}$ be the class of all edges incident to $v$. According to Lemma \ref{Lemma 1}(i), $\mid E_i \setminus \{v\}\mid \leq \sqrt{n}$, for each $i$. By hypothesis, $\mid \adj(v)\mid = n$, which implies that $\mid E_i\setminus\{v\}\mid = \sqrt{n}$. Thus, $\mid E_i\mid=\sqrt{n}+1$ and, by Lemma \ref{Lemma 1}(ii), each vertex of $E_i$ has degree $\sqrt{n}$. Thus, each $E_i$ can be represented as
$E_i=\{v,v_{i1},\ldots,v_{i\sqrt{n}}\}$, where $d(v_{ij})=\sqrt{n}$, and $\mid\adj(v_{ij})\mid = n$.
Since $d(v_{ij})=\sqrt{n}$, and $E_i$ is incident to $v_{ij}$, the class of all edges incident to $v_{ij}$ can be expressed as $\{E_i,E_{ij2},\ldots,E_{ij\sqrt{n}}\}$. In addition, the fact that $\mid \adj(v_{ij})\mid=n$, combined with Lemma \ref{Lemma 1}(i), implies that
\[\mid E_i\setminus\{v_{ij}\}\mid = \mid E_{ij2}\setminus\{v_{ij}\}\mid = \cdots =\mid E_{ij\sqrt{n}}\setminus\{v_{ij}\}\mid = \sqrt{n}.\]
By the linearity of $\mathscr{H}$, if $1\leq j<k\leq \sqrt{n}$, the class of the $\sqrt{n}$ edges incident to $v_{ij}$, and the class of the $\sqrt{n}$ edges incident to $v_{ik}$, share exactly one edge (namely, $E_i$). Hence, the number of edges of $\mathscr{E}$, incident to one of $v_{i1},\ldots,v_{i\sqrt{n}}$ is $\sqrt{n}\sqrt{n} - (\sqrt{n}-1)$. And the number of edges incident to one of $v,v_{i1},\ldots,v_{i\sqrt{n}}$ is $\sqrt{n}\sqrt{n} - (\sqrt{n}-1) + (\sqrt{n}-1) = n$. Since $\mid \mathscr{E}\mid\leq n$, we must have that $\mid \mathscr{E}\mid = n$, which proves (i). 

In addition,
\[\mathscr{E} = \bigcup\limits_{j=1}^{\sqrt{n}} \{E_i,E_{ij2},\ldots,E_{ij\sqrt{n}}\}.\]
Since $\mid E_i\mid = \mid E_{ij2}\mid = \cdots = \mid E_{ij\sqrt{n}}\mid = \sqrt{n}+1$, we conclude that $\mid E \mid = \sqrt{n}+1$, for all $E\in \mathscr{E}$, which proves (ii).

Let $u\in\mathscr{V}$. If for some $i$, $u\in E_i$, then $d(u) = \sqrt{n}$. Now, suppose that $u\in \mathscr{V}\setminus \bigcup\limits_{i=1}^{\sqrt{n}} E_i$. Then $u$ is not adjacent to $v$. Since each edge incident to $u$ must also be incident to one of $v_{i1},\ldots,v_{i\sqrt{n}}$, and given that no $v_{ij}$ can be shared by two edges incident to $u$, it follows that $d(u) \leq \sqrt{n}$. But, by hypothesis, $d(u)\geq \sqrt{n}$. Therefore, $d(u)=\sqrt{n}$, which proves (iii).

Let $E\in \mathscr{E}$. By (ii), $\mid E\mid = \sqrt{n}+1$, and by (iii), each of the $\sqrt{n}+1$ vertices of $E$ has degree $\sqrt{n}$. Therefore, the number of edges incident to at least one vertex of $E$ is $(\sqrt{n}+1)\sqrt{n} - \sqrt{n} = n$. In other words, each of the $n$ edges of $\mathscr{E}$ intersects $E$ and, by linearity, no edge in $\mathscr{E}\setminus \{E\}$ can intersect $E$ in more than one vertex. Hence, if $E'\in \mathscr{E}\setminus\{E\}$, $E$ and $E'$ must share exactly one vertex, which proves (iv).

Finally, if $E\in\mathscr{E}$, and $u\in \mathscr{V}\setminus E$, each of the $\sqrt{n}$ edges incident to $u$ must intersect $E$ in exactly one vertex. By linearity, no vertex of $E$ can be shared by two edges incident to $u$. It follows that exactly $\sqrt{n}$ of the $\sqrt{n}+1$ vertices of $E$ are adjacent to $u$, and exactly one of them is not, which proves (v).
\end{proof}

\begin{theorem}\label {Theorem 1}
Let $\mathscr{H} = (\mathscr{V},\mathscr{E})$ be a linear hypergraph with at most $n$ edges, each with at most $n$ vertices. Suppose that $\delta(\mathscr{H})\geq \sqrt{n}$. Suppose, in addition, that there is a vertex $v\in\mathscr{V}$, such that $d(v) = \sqrt{n}$, and $\mid\adj(u)\mid = n$, for all $u\in \adj(v)\cup\{v\}$. Then $\mathscr{H}$ is $n$-colorable.
\end{theorem}

\begin{proof}
Let $E\in\mathscr{E}$. By Lemma \ref{Lemma 1'}(ii), $E$ has cardinality $\sqrt{n}+1$, and can be represented as
$E = \{v_0,\ldots,v_{\sqrt{n}}\}$. For each $i$, let 
\[E_i = \{u\in\mathscr{V}: u \text{ is not adjacent to }v_i\} \cup \{v_i\}.\]
We will prove that $\mathscr{V}$ is the disjoint union of $E_0,\ldots,E_{\sqrt{n}}$. This amounts to showing that each vertex of $\mathscr{V}$ belongs to exactly one of $E_0,\ldots,E_{\sqrt{n}}$. Let $u\in E$. Then $u = v_i$ for some $i$, and thus, $u\in E_i$. In addition, for all $j\neq i$, $u\notin E_j$, for $v_i$ is adjacent to $v_j$. Now, let $u\in \mathscr{V}\setminus E$. By Lemma \ref{Lemma 1'}(v), there is exactly one vertex of $E$ not adjacent to $u$; call it $v_i$. Hence, $u\in E_i$, and $u\notin E_j$ for all $j\neq i$.

Consider the function $f:\mathscr{V}\rightarrow\{0,\ldots,\sqrt{n}\}$, defined by $f(u) = \sum\limits_{i=0}^{\sqrt{n}} i \chi_{E_i}(u)$ (where $\chi_{E_i}:\mathscr{V}\rightarrow \{0,1\}$ is the characteristic function of $E_i$). We will show that $f$ is a $(\sqrt{n}+1)$-coloring of $\mathscr{H}$. Suppose that two different vertices $u,u'$ have the same image under $f$, say $f(u) = f(u') = i$. Then, neither $u$ nor $u'$ is adjacent to $v_i$, and at least one of them, say $u$, is different from $v_i$. Let $D$ be an edge incident to $u$. Since $v_i\notin D$, it follows from Lemma \ref{Lemma 1'}(v) that exactly one vertex of $D$ (namely, $u$) is not adjacent to $v_i$. Therefore, $u'\notin D$. We have proven that no edge contains $u$ and $u'$ simultaneously; that is, $u$ is not adjacent to $u'$. We conclude that $f$ is a $(\sqrt{n}+1)$-coloring and hence, an $n$-coloring of $\mathscr{H}$.
\end{proof}

\textit{Note}: the proof of Theorem \ref{Theorem 1} shows that $\mathscr{H}$ admits a $(\sqrt{n}+1)$-coloring and, since each edge of the hypergraph has $\sqrt{n}+1$ vertices, $\chi(\mathscr{H})=\sqrt{n}+1$.

\begin{theorem}\label{Coloring}
Let $\mathscr{H} = (\mathscr{V},\mathscr{E})$ be a linear hypergraph with at most $n$ edges, each with at most $n$ vertices. Suppose that $\delta(\mathscr{H})\geq \sqrt{n}$. Then $\mathscr{H}$ admits an $n$-coloring.
\end{theorem}

\begin{proof}
Suppose that there is a vertex $v\in\mathscr{V}$, such that $d(v) = \sqrt{n}$, and $\mid \adj(u)\mid = n$, for all $u\in \adj(v) \cup \{v\}$. Then, by Theorem \ref{Theorem 1}, $\mathscr{H}$ is $n$-colorable. Now, suppose that no such $v$ exists.
Let $\mathscr{V} = \{v_1,\ldots,v_k\}$, where the vertices of $\mathscr{V}$ are ordered in such a way that 
\begin{enumerate}[(i)]
\item If $i<j$, then $d(v_i)\geq d(v_j)$; and
\item if $i<j$, and $d(v_i) = d(v_j)$, then $\mid \adj(v_i)\mid \geq\mid \adj(v_j)\mid$.
\end{enumerate}
We will show that $\mathscr{H}$ can be colored by induction. Color $v_1$ with any color. Let $2\leq i\leq k$, and suppose that $v_1,\ldots, v_{i-1}$ have been colored. We will consider two cases: (a) $d(v_i)>\sqrt{n}$; (b) $d(v_i) = \sqrt{n}$.\\
(a) If $E$ is an edge incident to $v_i$, and $v\in E$ is a vertex that has been colored, there is an edge (namely, $E$) incident to both $v_i$ and $v$, and there are $d(v) - 1$ ($\geq d(v_i)-1$) edges containing $v$ but not $v_i$. Note that these $d(v)-1$ edges are among the $n-d(v_i)$ edges not incident to $v_i$. Hence, the number of vertices of $E$ that have been colored is bounded above by $\dfrac{n-d(v_i)}{d(v_i) - 1}$. Since the same is true for each edge incident to $v_i$, the number of vertices adjacent to $v_i$ that have been colored is at most 
$d(v_i)\dfrac{n-d(v_i)}{d(v_i)-1}$. Since $d(v_i)>\sqrt{n}$, $d(v_i)\dfrac{n-d(v_i)}{d(v_i) - 1}<n$. Therefore, $v_i$ can be colored with one of the remaining $n-d(v_i)\dfrac{n-d(v_i)}{d(v_i) - 1}$ colors. \\
 (b) We now consider the case $d(v_i) = \sqrt{n}$. If $\mid\adj(v_i) \mid= k<n$, then at most $k$ colors have been used to color $k$ vertices of $\adj(v_i)$, and $v_i$ can be colored with any of the remaining colors. On the other hand, if $\mid\adj(v_i)\mid = n$, there must be a vertex $u$ adjacent to $v_i$, with $\mid\adj(u)\mid<n$. In addition, the fact that $d(v_i) = \sqrt{n}$, and $\mid\adj(v_i)\mid = n$, implies that each edge incident to $v_i$ must have cardinality $\sqrt{n}+1$, and by Lemma \ref{Lemma 1}(ii), each vertex adjacent to $v_i$ must have degree $\sqrt{n}$. Since $d(v_i) = d(u)$, and $\mid\adj(v_i)\mid>\mid\adj(u)\mid$, it follows that $u$ has not yet been colored. We conclude that the number of vertices adjacent to $v_i$ that have been colored is strictly less than $n$, and thus, $v_i$ can be colored with one of the remaining colors.
\end{proof}

\begin{definition}
A hypergraph $\mathscr{H} = (\mathscr{V},\mathscr{E})$, with $\mid \mathscr{E} \mid = n$, is called \textbf{weakly dense}, if no integer $k$ in the half-open interval $[2,\sqrt{n})$ is the degree of more than $k^2$  vertices; that is, for all $k \in[2,\sqrt{n})$, there are at most $k^2$ vertices with degree $k$.
\end{definition}

\begin{theorem}[The EFL Conjecture for Weakly Dense Hypergraphs]\label{EFL}
Let $\mathscr{H}$ be a linear $n$-uniform hypergraph with $n$ edges. If $\mathscr{H}$ is weakly dense, then $\chi(\mathscr{H})=n$.
\end{theorem}

\begin{proof}
Denote $\mathscr{H} = (\mathscr{V},\mathscr{E})$, and let $\mathscr{V} = \mathscr{V}_1 \cup \mathscr{V}_2 \cup \mathscr{V}_3$, where $\mathscr{V}_1 = \{v\in \mathscr{V}: d(v)\geq \sqrt{n}\}$; $\mathscr{V}_2 = \{v\in\mathscr{V}:2\leq d(v)<\sqrt{n}\}$; and $\mathscr{V}_3 = \{v\in \mathscr{V}: d(v) = 1\}$. 
We will construct an $n$-coloring $f$ of $\mathscr{H}$ as follows. 
Define $\mathscr{E}_1 = \{E\cap\mathscr{V}_1: E\in\mathscr{E}$, and $E\cap \mathscr{V}_1\neq\varnothing\}$, and let $\mathscr{H}_1 = (\mathscr{V}_1,\mathscr{E}_1)$. Note that $\mathscr{H}_1$ is a linear hypergraph with at most $n$ edges, each with at most $n$ vertices, such that $d(v)\geq \sqrt{n}$, for all $v\in\mathscr{V}_1$. By Theorem \ref{Coloring}, $\mathscr{H}_1$ admits an $n$-coloring $f_1$. Note that $f_1$ is also an $n$-coloring of  $\mathscr{V}_1$ in $\mathscr{H}$. Define $f(v) = f_1(v)$, for all $v\in \mathscr{V}_1$. We now define $f(v)$, for all $v\in\mathscr{V}_2$, as follows. Denote $\mathscr{V}_2 = \{v_1,\ldots,v_s\}$, where $d(v_i)\geq d(v_j)$ if $i< j$. Let $1\leq r\leq s$, and suppose that $v_1,\ldots,v_{r-1}$ have been colored (i.e., $f(v_1),\ldots,f(v_{r-1})$ have been defined). Let $E_1,\ldots,E_{d(v_{r})}$ be the edges of $\mathscr{E}$ that are incident to $v_r$. For each $1\leq j\leq d(v_r)$, let $i_j$ be the number of vertices in $E_j\setminus \{v_r\}$ with degree $d(v_r)$. Let $i=i_1+\cdots+i_{d(v_r)}$. Suppose that $d(v)>d(v_r)$, for some $v\in E_j$. Note that there is one edge (namely, $E_j$) that contains $v$, as well as all vertices in $E_j$ of degree $d(v_r)$, and there are $d(v) - 1$ ($\geq d(v_r)$) edges that contain $v$ but not any of the vertices of $E_j$, of degree $d(v_r)$. These $d(v) - 1$ edges are among the $(n-d(v_r))-i_j(d(v_r) - 1)$ edges not containing any of the vertices of $E_j$, of degree $d(v_r)$. Therefore, the number of vertices in $E_j$, of degree larger than $d(v_r)$ is at most
\[\dfrac{(n-d(v_r)) - i_j (d(v_r) - 1)}{d(v_r)}.\] 
It follows that the number of vertices in $E_j$ that have already been colored is at most
\[i_j + \dfrac{(n-d(v_r)) - i_j (d(v_r) - 1)}{d(v_r)},\]
and the number of vertices of $\mathscr{V}$, adjacent to $v_r$, that have been colored is at most
\begin{align*}
&\sum\limits_{j=1}^{d(v_r)} \left[[i_j + \dfrac{(n-d(v_r)) - i_j (d(v_r) - 1)}{d(v_r)}\right] =\\
&\sum\limits_{j=1}^{d(v_r)} i_j + \sum\limits_{j=1}^{d(v_r)}\dfrac{n-d(v_r)}{d(v_r)} - \sum\limits_{j=1}^{d(v_r)}\dfrac{i_j  (d(v_r)-1)}{d(v_r)} = \\
& i+ (n-d(v_r)) - \dfrac{i(d(v_r) - 1)}{d(v_r)} =\\
& n-d(v_r) + \dfrac{i}{d(v_r)}.
\end{align*}
Since $\mathscr{H}$ is weakly dense, $i<d^2(v_r)$, and thus, $n-d(v_r)+\dfrac{i}{d(v_r)}<n$. This implies that the number of vertices adjacent to $v_r$ that have been colored is less than $n$. Define $f(v_r)$ as one of the remaining colors. By recurrence, we can color all vertices in $\mathscr{V}_2$. Finally, we color the vertices of $\mathscr{V}_3$ as follows. For each $E\in\mathscr{E}$, let $k_E$ be the number of vertices of $E$ that have been colored. If a vertex $v\in E$ has not been colored, it means that $E$ is the only edge incident to $v$. Hence, the $n-k_E$ vertices of $E$ not yet colored can be arbitrarily colored with the $n-k_E$ remaining colors. Thus, $f$ is an $n$-coloring of $\mathscr{H}$ and, since each edge contains $n$ vertices, $\chi(\mathscr{H})=n$.
\end{proof}

\noindent \textbf{Acknowledgements}: A big thanks to my dear wife Danisa for her support, and for typing this paper. In these excruciating times, her words of encouragement have been an invaluable source of strength.

\end{document}